\documentclass[12pt]{amsart}
\usepackage{amsfonts}
\usepackage{latexsym,amscd,amssymb}
\usepackage{movie15}
\usepackage{amsmath}
\usepackage[all]{xy}
\usepackage{array}
\usepackage{hyperref}
\usepackage{graphicx}
\usepackage{epsfig}

\newtheorem*{theorem*}{Theorem}
\newtheorem{theorem}{Theorem}[section]
\newtheorem{proposition}[theorem]{Proposition}
\newtheorem{lemma}[theorem]{Lemma}
\newtheorem{corollary}[theorem]{Corollary}
\newtheorem{definition}[theorem]{Definition}

\newtheorem*{lemma*}{Lemma}
\newtheorem*{proposition*}{Proposition}
\newtheorem*{corollary*}{Corollary}

\numberwithin{equation}{section}

\allowdisplaybreaks

\begin{document}

\title{Convergence of Yang-Mills-Higgs flow for twist Higgs pairs on Riemann surfaces}

\author{Wei Zhang}

%School of Mathematical Sciences, University of Science and
%Technology of China

\thanks{Mathematics Classification Primary(2000): Primary 58E15, Secondary 53C07.\\
\indent The author is supported by the NSFC No. 11101393 and Fundamental Research Funds for the Central Universities of China WK0010000008.}
\maketitle
\begin{abstract}

We consider the gradient flow of the Yang-Mills-Higgs functional of
twist Higgs pairs on a Hermitian vector bundle $(E,H_0)$ over a
Riemann surface $X$. It is already known the gradient flow with
initial data $(A_0,\phi_0)$ converges to a critical point
$(A_\infty, \phi_\infty)$ of this functional. Using a modified
Chern-Weil type inequality, we prove that the limiting twist Higgs
bundle $(E, d_{A_\infty}'', \phi_\infty)$ is given by the graded
twist Higgs bundle defined by the Harder-Narasimhan-Seshadri
filtration of the initial twist Higgs bundle $(E,d_{A_0}'',\phi_0)$,
generalizing Wilkin's results for untwist Higgs bundle.

\end{abstract}

\section{Introduction}

Higgs bundle originates from Hitchin's reduction of self-dual
equation on $\mathbb{R}^4$ to Riemann surface(cf. \cite{Hit87self}),
constituted by a holomorphic vector bundle $E \rightarrow X$, and a
holomorphic (1,0)-form $\phi$ taking value in $End(E)$. If the base
manifold $X$ is a smooth Riemann surface, it is equivalent to say
$\phi \in H^0(X, End(E)\otimes K)$, where $K$ is the canonical line
bundle of $X$. This suggests us that $K$ can be replaced by any line
bundle. The definition of twist Higgs bundle follows

\begin{definition}
A twist Higgs bundle is a pair $(E,\phi)$ where $E$ is a rank $n$
holomorphic vector bundle over a complex manifold $X$, $\phi \in
\Omega^{1,0}(End(E)\otimes L)$ is the holomorphic Higgs field
twisted with line bundle $L$, where $L$ is any fixed holomorphic
line bundle.
\end{definition}

%Notice that the integrable condition $\phi \wedge \phi=0$ always
%holds on Riemann surface.

To emphasis on the holomorphic structure, Higgs bundle also can be
denoted as $(E,\bar{\partial}, \phi)$. If we take $L$ to be the
trivial line bundle, then it becomes the usual Higgs bundle. For
simplicity, in this article, Higgs bundle means twist Higgs bundle,
and we will specify the usual Higgs bundle as untwist Higgs bundle.

Twist Higgs bundles share lots of tributes with the untwist Higgs
bundles. One can define stability on twist Higgs bundle and the
Hitchin-Kobayashi correspondence still holds(cf.
\cite{bradlow1995stable}). The twist Higgs bundle also admit the so
called Harder-Narasimhan-Seshadri(short as HNS) filtration(cf.
\cite{HT03mirror}). While, there are someting different, for
example, the moduli space of twist Higgs bundle is just a coarse
moduli(cf. \cite{nit91moduli}) rather than a fine moduli as the
usual Higgs bundle(cf. \cite{Hit87self}).

%rather than consider a single Higgs bundle,

Restrict ourselves to the case $X$ is a Riemann surface, fix a
$C^\infty$ complex vector bundle $E$ of rank $n$ with a Hermitian
metric $H$ and a holomorphic line bundle $L$ with Hermitian metric
$h$. Let $\mathcal{A}$ denote the space of connections on $E$
compatible with the metric. Notice that $\mathcal{A}$ is isomorphic
to the space $\mathcal{A}^{0,1}$, the space of holomorphic
structures on $E$. A pair $(A, \phi) \in \mathcal{A} \otimes
\Omega^{1,0}(End(E)\otimes L)$ is called a Higgs pair if
$d_A''\phi=0$ is satisfied, where $d_A''$ is the naturally induced
covariant derivative on $\Omega^{1,0}(End(E)\otimes L)$. Thus each
Higgs pair will endow $E$ a structure of Higgs bundle
$(E,d_A'',\phi)$. In \cite{Wil06morse}, Wilkin studied the
Yang-Mills-Higgs(short as YMH) flow of untwist Higgs pair over
Riemann surface, proved that the flow converges to the graded object
associated to the Harder-Narasimhan-Seshadri filtration. In this
article, we generalize the result to twist Higgs pair. The
convergence of the flow is analogous to Wilkin's case, or can be
viewed as a special case of Yue Wang and Xi Zhang's work(cf.
\cite{WZ11twisted}). So we focus on the asymptotic behavior of the
heat flow for Higgs pairs. Our first main theorem asserts that the
gradient flow preserve the Harder-Narasimhan(short as HN) type. The
key point in the proof is the Chern-Weil formula(cf.
\cite{Sim88constructing}) for subbundle $S$ defined by projection
$\pi$
\begin{equation*}
deg(S)=\frac{1}{2\pi}\int_X(Tr(\sqrt{-1}\Lambda F_A \pi)-|d_A'' \pi|^2)dvol.
\end{equation*}
Using this formula, we modify the method in \cite{DW2004} to control
the degree of twist Higgs sub-bundle, and show that the HN type is
nondecreasing along the flow. Following the idea of Atiyah and
Bott(cf. \cite{AB83}), Wilkin employed all the convex invariant
function on the Lie algebra to identify the HN type of the Higgs
bundle. His method is still available for twist Higgs bundle on
Riemann surface, but we take the method of Daskalpoulos and
Wentworth(cf. \cite{DW2004}), using merely a subclass of convex
functional the so called weighted YMH functional, which is easy to
be generalized to the Higgs bundle on K{\"a}hler surface(cf.
\cite{LZ2011}). Combine the non-increasing of the weighted YMH
functional and the non-decreasing of the HN type along the gradient
flow, then use the so-called approximate Hermitian structure to
eliminate the possibility of jumping phenomenon, we get our first
main theorem.

\begin{theorem}\label{th:preserveHN}
Let $(A_t, \phi_t)$ be a smooth solution of the gradient flow on the Hermitian vector bundle $(E,H)$ with initial condition $(A_0, \phi_0)$ and $(A_\infty,\phi_\infty)$
be the limit. Then the Harder-Narasimhan type of $(E,d_{A_\infty}'', \phi_\infty)$ is the same as that of $(E,d_{A_0}'', \phi_0)$.
\end{theorem}

Following Donaldson\cite{Don85anti}, we constructs a nontrivial
holomorphic map $(E,d_{A_0}'', \phi_0) \rightarrow
(E,d_{A_\infty}'', \phi_\infty)$. With such a map in hand, one may
then apply the basic principle that a nontrivial holomorphic map
between stable bundles of the same rank and degree must be an
isomorphism. Denote the graded object associated to the HNS
filtration by $Gr^{hns} (E,d_{A_0}'', \phi_0)$, there is the second
main theorem

\begin{theorem}\label{th:convergeGR}
The Higgs bundle $(E,d_{A_\infty}'', \phi_\infty)$ is
holomorphically isomorphic to the graded object $Gr^{hns}
(E,d_{A_0}'', \phi_0)$.
\end{theorem}

One may make further investigation on the moduli space of twist
Higgs bundle over Riemann surface, getting the stratification
structure according to the HN type. The discussion is totally
parallel with the untwist case in \cite{Wil06morse}. Recently, there
are several excellent works about the convergence of the Yang-Mills
flow(cf.
\cite{Sibley2012YangMills},\cite{CollinsJacob2012YangMills}) on
manifold with dimension greater than two. We hope that our results
on twist Higgs bundle could be generalized to higher dimension.

%and Higgs-Yang-Mills(cf. \cite{LiZhang2012HYM})

This article is organized as follows. In section 2, we collect some
preliminary material about Higgs bundle, such as the Hitchin's
equation, stability, Hitchin-Kobayashi correspondence, HNS
filtration, the Yang-Mills-Higgs functional of Higgs pair and its
gradient flow. In section 3, we focus on the Harder-Narasimhan type
of the limit of the gradient flow. Using tools like weighted YMH
functional and approximate critical Hermitian structure, we prove
the first main theorem. Finally, in section 4, we get the proof for
our second main theorem.

\section{Preliminary}

\subsection{Higgs bundle}

Fix a Hermitian metric $h$ on the holomorphic line bundle $L$ once
for all, there is a unique Chern connection compatible with $h$ and
the holomorphic structure.

Endow Hermitian metric $H$ on $E$, denote $A=(\bar{\partial},H)$ the
connection 1-form of Chern connection respect to $H$ s.t.
$d^{\prime\prime}_A=\bar{\partial}$, and $F_{(\bar{\partial},H)}$
the curvature two form. Combining the fixed connection on $L$, there
is induced connection on $End(E)\otimes L$, still denoted as $d_A$.
If $\phi=\Phi\mathrm{d}z \otimes s$, $\Phi \in \Gamma(End(E))$, $s
\in \Gamma(L)$, then $\phi^{*H}$ is set to be $\Phi^{*H}
\mathrm{d}\bar{z} \otimes h(s)$, here $\Phi^{*H}$ is the adjoint of
$\Phi$ under $H$, $h(s)$ is the section of $L^*$ defined by $h(s,\
)$. Moreover, we define $[\phi, \phi^{*H}] \in End(E)$ as the Lie
bracket extended to $End(E)\otimes L$ and $End(E)\otimes L^*$ valued
1-form, means $\phi\phi^{*H}+\phi^{*H}\phi$, where the contraction
of $L$ with $L^*$ is taken place at the same time.

There is the so called Hitchin's equation on the Hermitian metric
$H$ over twist Higgs bundle $(E, \bar{\partial}, \phi)$.
\begin{equation}\label{eq:hitchinMetric}
\sqrt{-1}\Lambda_\omega(F_{(\bar{\partial},H)}+[\phi, \phi^{*H}])=\mu Id_E
\end{equation}
where $\omega$ is any fixed K{\"a}hler form(the equation is
conformal invariant, independent on the metric on $X$, so sometimes
we neglect the subscript $\omega$). For $[\phi, \phi^{*H}]$ is
always traceless, $\mu=\frac{deg E}{rank E}$ is the slope of the
vector bundle.

A natural question is when a Higgs bundle admits a solution of
Hitchin's system. More preciously, for a fixed holomorphic bundle
$(E,\bar{\partial})$ and holomorphic $\phi \in
\Omega^{1,0}(End(E)\otimes L)$, find a Hermitian metric $H$, s.t.
the Chern connection $(\bar{\partial},H)$ and $\phi$, $\phi^{*H}$
satisfying the equation. This turns out to closely relate to the so
called stability. The Hitchin-Kobayashi correspondence is totally
parallel with the untwist situation,

\begin{definition}
A Higgs bundle $(E,\bar{\partial},\phi)$ is called stable(semi-stable) if for any $\phi$ invariant\footnote{means that $\phi(F) \subset F\otimes L \otimes K$} holomorphic subbundle $F$, $\frac{deg F}{rank F} <(\leq) \frac{deg E}{rank E}$. Moreover, the Higgs bundle is called polystable if $(E,\phi)=(E_1,\phi_1)\oplus\cdots\oplus (E_r,\phi_r)$, here $\phi_i \in \Omega^{1,0}(End(E_i)\otimes L)$ and $(E_i,\phi_i)$ are all stable and with same slope.
\end{definition}

and

\begin{theorem}[\cite{Hit87self},\cite{Sim88constructing}]
A Higgs bundle $(E,\bar{\partial},\phi)$ admit a Hermitian metric $H$ satisfying the Hitchin equation if and only if it is polystable.
\end{theorem}
Readers can refer to \cite{Hit87self},\cite{Sim88constructing} for the proof of the untwist case, and \cite{bradlow1995stable} for the twist case.

If we restrict ourselves to the untwist Higgs bundle, consider the
new $GL(n,\mathbb{C})$ connection
$\nabla_{(\bar{\partial},\phi,H)}=d_A+\phi+\phi^{*H}$, denote its
curvature  by $R_{(\bar{\partial},\phi,H)}$ to distinguish with
$F_{(\bar{\partial},H)}$. Although
$\nabla_{(\bar{\partial},\phi,H)}$ is not a metric connection to
$H$,
\begin{equation*}
\begin{split}
R_{(\bar{\partial},\phi,H)}=&\nabla_{\bar{\partial},\phi,H}^2=d_A^2+d_A(\phi+\phi^{*H})+[\phi, \phi^{*H}]\\
=&F_{(\bar{\partial},H)}+[\phi, \phi^{*H}]+d^{\prime\prime}_A
\phi+d^\prime_A \phi^{*H}.
\end{split}
\end{equation*}
This means that if $H$ is a solution to Hitchin's equation of the
untwist Higgs bundle, then the traceless part of the curvature
satisfies $R_{(\bar{\partial},\phi,H)}^\bot=0$.

In the twist case, we can not produce the non-unitary connection
from the Higgs data, but the behavior of
$F_{(\bar{\partial},H)}+[\phi, \phi^{*H}]$ is somehow similar to
$R_{(\bar{\partial},\phi,H)}$, so we denote
$F_{(\bar{\partial},H)}+[\phi, \phi^{*H}]$ as
$\Theta_{(\bar{\partial},\phi,H)}$ for the twist Higgs bundle to
prevent ambiguity. Sometimes we will abbreviate some subscripts as
$\Theta_H$ if only the metric $H$ varies.

\subsection{Harder-Narasimhan and Seshadri filtrations}

Given arbitrary twist Higgs bundle $(E, \bar{\partial}, \phi)$, it
may not be stable or semitable, but we have

\begin{lemma}
Let $(E, \bar{\partial}, \phi)$ be a twist Higgs bundle, then there
is a unique filtration, called Harder-Narasimhan filtration of $E$
by $\phi$-invariant holomorphic subbundles $0=E_0 \subset E_1
\subset \cdots \subset E_r=E$, s.t. $F_i=E_i/E_{i-1}$ is Higgs
semistable with respect to the quotient Higgs field $\tilde{\phi}_i
\in \Omega^{1,0}(End(F_i)\otimes L)$ for all $i$ and $\mu(F_1)>
\cdots
>\mu(F_r)$.
\end{lemma}

And the $n$-tuple $\overrightarrow{\mu}=(\overbrace{\mu_1, \cdots,
\mu_1}^{rank E_1}; \overbrace{\mu_2, \cdots, \mu_2}^{rank E_2}
\cdots \overbrace{\mu_r, \cdots, \mu_r}^{rank E_r})$ is called the
type of the HN filtration. Similarly,

\begin{lemma}
Let $(E, \bar{\partial}, \phi)$ be a semi-stable Higgs bundle, then there is a filtration, called Seshadri filtration of $E$ by $\phi$-invariant holomorphic subbundles $0=E_0 \subset E_1 \subset \cdots \subset E_s=E$, s.t. $Q_i=E_i/E_{i-1}$ is Higgs stable for all $i$.
\end{lemma}

Notice that for $Q_i$ may have same slope, Seshadri filtration is
not unique. Combining above two filtrations together, we have

\begin{proposition}
Let $(E, \bar{\partial}, \phi)$ be a Higgs bundle, then there is a
double filtration $\{E_{i,j}\}$ of $E$ called $\phi$-invariant
Harder-Narasimhan-Seshadri filtration, s.t. $\{E_i\}_{i=1}^r$ is the
HN filtration, and $\{E_{i,j}\}_{j=1}^{s_i}$ is a Seshadri
filtration of $E_i/E_{i-1}$.
\end{proposition}

The associated graded object
\begin{equation*}
Gr^{HNS}(E,\bar{\partial},\phi)=\oplus^r_{i=1} \oplus^{s_i}_{j=1} Q_{i,j},
\end{equation*}
where $Q_{i,j}=E_{i,j}/E_{i,j-1}$, is uniquely determined by the
isomorphism class of $(E,\bar{\partial},\phi)$. It is easy to see
that $Gr^{HNS}(E,\bar{\partial},\phi)$ is not gauge equivalent to
$(E,\bar{\partial},\phi)$ except itself is stable. This provides us
an algebraic way to split a Higgs bundle into a direct sum of stable
Higgs bundles.

\subsection{Higgs-Yang-Mills flow}

On the other hand, rather than fixing the Higgs bundle
$(E,\partial,\phi)$ to find the Hermitian Yang-Mills metric $H$, we
fix the $C^\infty$ bundle $E$ and a Hermitian metric $H$ on it, to
find a (integrable)connection compatible with the metric and a $\phi
\in \Omega^{1.0}(End(E)\otimes L)$, such that, firstly, $E$ is a
holomorphic vector bundle\footnote{This always holds on Riemann
surface} with holomorphic structure $d_A''$; Secondly, $\phi$ is
holomorphic under $\bar{\partial}=d_A''$, i.e. a Higgs field, thus
the Higgs pair $(A,\phi)$ makes $(E,d_A'',\phi)$ a Higgs bundle.
Thirdly, $H$ satisfies the Hitchin's equation. Summing up, for fixed
$H$, there is a equation on the Higgs pair $(A,\phi)$
\begin{equation}\label{eq:hitchinPair}
\begin{cases}
F_A+[\phi, \phi^{*H}]=-\sqrt{-1}\mu Id_E \omega\\
d^{\prime\prime}_A \phi=0\\
\end{cases}
\end{equation}
we still call it the Hitchin's equation. In this turn, we denote
$F_A+[\phi, \phi^{*H}]$ as $\Theta_{(A,\phi)}$, and we will omit the
subscript if it does not cause confusion. The solutions of this
equation can be interpreted in Morse theory. Consider the
Yang-Mills-Higgs functional on $(A,\phi)$ restricting to the level
set $d^{\prime\prime}_A \phi=0$
\begin{equation}\label{eq:hymfunc}
\text{YMH}(A,\phi)=||F_A+[\phi,\phi^{*H}]||^2=\int |\Lambda_\omega
(F_A+[\phi,\phi^{*H}])|^2 \mathrm{d} vol,
\end{equation}
the solutions of Hitchin equation is the local minimum, a subclass
of the critical points of this functional. To find all the critical
points, we consider the associated gradient flow
\begin{equation*}
\begin{cases}
\frac{\partial A}{\partial t}=*d_A*(F_A+[\phi,\phi^{*H}])=-d_A^*\Theta\\
\frac{\partial \phi}{\partial t}=*[\phi, *(F_A+[\phi,\phi^{*H}])]=*[\phi,*\Theta]
\end{cases}
\end{equation*}
Since both the holomorphic structure and Hermitian metric are fixed,
the connection $A$ is totally determined by its $(0,1)$ part $A''$,
the first equation is equivalent to
\begin{equation*}
\frac{\partial A''}{\partial t}=*d''_A*(F_A+[\phi,\phi^{*H}])=-d_A''^*\Theta.
\end{equation*}

Before solving this evolution equation system for any initial data,
we should notice that $*$ acting on $1$-form amounts to multiplying
the complex number $i$, and
\begin{equation}\label{eq:gradientflow}
\begin{split}
\frac{\partial d_A''\phi}{\partial t}=&\frac{\partial}{\partial t}(d'' \phi+ [A'',\phi])
=d_A''\frac{\partial \phi}{\partial t}+[\frac{\partial A''}{\partial t}, \phi]\\
=&i[d_A''*\Theta, \phi]+id_A''[\phi,*\Theta]\\
=&i((d_A''\Theta)\phi+\phi d_A''\Theta+d_A''(\phi \Theta)-d_A''(\Theta\phi))\\
=&i((d_A''\phi)\Theta-\Theta(d_A''\phi))=i[d_A''\phi,\Theta]
\end{split}
\end{equation}
Thus the holomorphicity of Higgs field $d_A''\phi=0$ is preserved by
the gradient flow, it makes sense to restrict on the level set
$d_A''\phi=0$ to solve the Cauchy problem of the gradient flow.

To get the existence and convergence properties of this gradient
flow, by Simpson\cite{Sim88constructing}, one fixes $(A_0,\phi_0)$,
letting $H$ change along the following heat equation
\begin{equation}\label{eq:heatflow}
H^{-1}\frac{\partial H}{\partial t}=-i \Lambda \Theta^\bot_H.
\end{equation}
If $H(t)$ is the solution for this equation, then there is a gauge
transformations $g(t)$ determined by $H(t)$, s.t. $(A(t),
\phi(t))=(g(t)\cdot A_0, g(t)\cdot \phi_0)$ will be a solution to
Equation \eqref{eq:gradientflow}(the explicit expression of the
gauge transformation can be found in \cite{Wil06morse}). In the
untwist case, Simpson had proved that solution to Equation
\eqref{eq:heatflow} exists for all time and depends continuously on
the initial condition $H(0)$. The twist case can be viewed as a
special case of \cite{WZ11twisted}. Via the equivalence of above
heat flow and the gradient flow of YMH, Wilkin(cf.
\cite{Wil06morse}) proved the following properties of the solution
to Equation \eqref{eq:gradientflow}(the proof in twist case is
identical).

$\bullet$ Existence for all time and uniqueness.

$\bullet$ Convergence modulo gauge transformation.

$\bullet$ Convergence without gauge transformation.

$\bullet$ Continuous dependence on initial condition for any fixed
$T<\infty$ in the $H^k$ norm, for any $k \in \mathbb{N}$.

If the initial data $(A_0,\phi_0)$ define a stable Higgs bundle
$(E,d_{A_0}'',\phi_0)$, then the limit $(A_\infty, \phi_\infty)$
will satisfy Equation \eqref{eq:hitchinPair}, i.e. there is a gauge
transformation relate $(A_0,\phi_0)$ to $(A_\infty, \phi_\infty)$.
Without any stable assumption on the initial data, $(A_\infty,
\phi_\infty)$ may not be a solution of the Hitchin's equation. There
should be a precise description of the limit.

\begin{proposition}
Let $(A,\phi)$ be a critical point of the YMH functional, then there is an $\phi$-invariant orthogonal splitting $(E, d_A'', \phi)=\oplus^l_{i=1}(E_i, d_{A_i}'', \phi_i)$, s.t.
\begin{equation*}
\sqrt{-1}\Lambda \Theta_i=\mu_i Id_{E_i}
\end{equation*}
where $\Theta_i=F_{A_i}+[\phi_i,\phi_i^{*H}]$ and $\mu_i=\mu(E_i)$.
\end{proposition}
We only sketch the proof. The critical points of Equation
\eqref{eq:hymfunc} satisfying the Euler-Lagrange equations $d_A
*\Theta=0$ and $[\phi,
*\Theta]=0$. The first equation implies the eigenvalues of $*\Theta$
are all constant hence inducing a splitting of the vector bundle.
While second equation shows that this splitting is $\phi$-invariant.
So this analytic limit splits into direct sum of polystable Higgs
bundle.

Recall for a twist Higgs bundle $(E,\bar{\partial},\phi)$, there is
a graded Higgs bundle obtained via HNS filtration. Endow any
Hermitian metric $H$ on this twist Higgs bundle, denote the
compatible Chern connection as $A$. Forget the holomorphic structure
on $E$, consider the YMH flow for this twist Higgs pair $(A,\phi)$
on the $C^\infty$ bundle $E$, there is also a split bundle at the
limit. We want to show that this two kinds of splitting coincide.
The proof is divided into two steps. We first show the gradient flow
keep the HN type. Secondly, we show the limit of the gradient flow
must be the graded object defined by the HNS filtration.

\section{Harder-Narasimhan type of the limit}
The solution of the YMH flow in finite time equals to a gauge
transformation, so the jumping phenomenon of the HN type only takes
place at the limit. We try to relate the HN type with the weighted
YMH functionals, and use these functionals to identify the HN type.
\subsection{YMH functional and HN type}
Recall some basic facts about the YMH functional and HN type without
proof.
\begin{proposition}
Let $(A_t,\phi_t)$ be a solution of Equation \eqref{eq:gradientflow}, then
\begin{equation*}
\frac{\partial}{\partial t}|\Lambda \Theta|^2+\triangle_A |\Lambda \Theta|^2 \leq 0.
\end{equation*}
where $\triangle_A$ is the Hodge Laplace of $d_A$. Furthermore,
integrate the above expression, there is $\frac{d}{d
t}||\Theta||^2=-2 ||d_A^* \Theta||^2 \leq 0$, i.e. $t \rightarrow
\text{YMH}(A_t, \phi_t)$ is non-increasing.
\end{proposition}

The proof of the untwist case follows form \cite{Wil06morse}, the twist case is a special case of \cite{WZ11twisted}. By the convergence of the YMH flow, $\Lambda R_t \stackrel{L^p}\longrightarrow \Lambda R_\infty$, there is

\begin{lemma}
\begin{equation*}
\lim_{t\rightarrow\infty}\text{YMH}(A_t,\phi_t)=\text{YMH}(A_\infty,
\phi_\infty).
\end{equation*}
\end{lemma}

In order to compare different HN type, define a partial order of the
$n$-tuple $\overrightarrow{\mu}=(\mu_1,\cdots,\mu_n)$,
$\mu_1\geq\cdots\geq\mu_n$. For the Chern class is fixed, we only
need to take care the case
$\sum^n_{i=1}\mu_i=\sum^n_{i=1}\lambda_i$. We call
$\overrightarrow{\mu} \leq \overrightarrow{\lambda}$ if $\sum_{j
\leq k}\mu_j=\sum_{j \leq k}\lambda_j$ for all $k=1,\dots,n$. As we
know, if $E$ admit a critical twist Higgs pair, then it splits.
Abuse the notation, let $\overrightarrow{\mu}$ denote the split
bundle, then $\text{YMH}(\overrightarrow{\mu})=2\pi \sum^n_{i=1}
\mu_i^2$. It is easy to verify $\overrightarrow{\mu} \leq
\overrightarrow{\lambda}$ implying $\text{YMH}(\overrightarrow{\mu})
\leq \text{YMH}(\overrightarrow{\lambda})$. In the next, we study
how the HN type changes along the gradient flow. We need an
algebraic lemma.

\begin{lemma}
Let $(E, \bar{\partial},\phi)$ be a twist Higgs bundle and $S$ be a
$\phi$ invariant subbundle. Endow a Hermitian metric on $E$, let
$\pi = \pi^* =\pi^2$ denote the orthogonal projection onto the
subbundle $S$. Then
\begin{equation*}
Tr([\Phi,\Phi^*] \pi) =|[\phi, \pi]|^2.
\end{equation*}
where the inner product $|[\phi,\pi]|^2$ is defined to be
$Tr([\phi,\pi][\phi,\pi]^{*H})$ after contraction the section of $L$
and $L^*$ by the fix Hermitian metric $h$.
\end{lemma}

\begin{proof}
Compute it straight forward.
\begin{equation*}
\begin{split}
|[\phi,\pi]|^2=&Tr([\phi,\pi][\phi,\pi]^{*H})=Tr((\phi\pi-\pi\phi)(\pi\phi^{*H}-\phi^{*H}\pi))\\
=&Tr(\phi\pi\pi\phi^{*H}-\phi\pi\phi^{*H}\pi-\pi\phi\pi\phi^{*H}+\pi\phi\phi^{*H}\pi)
\end{split}
\end{equation*}
by the acyclicity of the trace, there is
$-\phi\pi\phi^{*H}\pi=\pi\phi^{*H}\pi\phi$ thus
\begin{equation*}
-\phi\pi\phi^{*H}\pi-\pi\phi\pi\phi^{*H}=[\pi\phi^{*H},\pi\phi]
\end{equation*}
which is always trace free. Still by the acyclicity,
\begin{equation*}
|[\phi,\pi]|^2=Tr(-\phi^{*H}\phi\pi\pi+\phi\phi^{*H}\pi\pi)=Tr([\phi,\phi^{*H}]\pi)
\end{equation*}\end{proof}

By Simpson(\cite{Sim88constructing}), the Chern-Weil formula reads
\begin{equation}\label{eq:chernweil}
\begin{split}
&deg(S)=\frac{1}{2\pi}\int_X Tr(\sqrt{-1}\Lambda F_A \pi)-|d_A'' \pi|^2 dvol\\
=&\frac{1}{2\pi}\int_X Tr(\sqrt{-1}\Lambda \Theta \pi)-Tr(\sqrt{-1}\Lambda [\phi,\phi^{*H}] \pi)-|d_A'' \pi|^2 dvol\\
=&\frac{1}{2\pi}\int_X Tr(\sqrt{-1}\Lambda \Theta \pi)dvol-\frac{1}{2\pi}||d_A'' \pi||^2-\frac{1}{2\pi}||[\phi,\pi]||^2\\
\end{split}
\end{equation}

\begin{proposition}
Denote the unitary gauge group of $E$ with fixed Hermitian metric
$H$ by $\mathfrak{u}(E)$. Let $(A_j, \phi_j)=g_j \cdot (A_0,
\phi_0)$ be a sequence of complex gauge equivalent Higgs structure
and $S$ be a $\phi_0$-invariant holomorphic subbundle of $(E,
d_{A_0}'', \phi_0)$ with rank $r$. Suppose $\sqrt{-1}\Lambda R_j
\stackrel{L^1}{\longrightarrow} \mathfrak{a}$, where $\mathfrak{a}
\in L^1(\sqrt{-1}\mathfrak{u}(E))$, and that the eigenvalues
$\lambda_1 \geq \cdots, \geq \lambda_n$ of $\mathfrak{a}$(counted
with multiplicities) are constant. Then $deg(S) \leq \sum_{i\leq r}
\lambda_i$.
\end{proposition}

\begin{proof}
Let $\pi_j: E\rightarrow g_j(s)$ denote the orthogonal projection.
By above Chern-Weil formula
\begin{equation*}
\begin{split}
deg(S)=&\frac{1}{2\pi}\int_X Tr(\sqrt{-1}\Lambda \Theta_j \pi_j)dvol-\frac{1}{2\pi}||d_{A_j}'' \pi_j||^2-\frac{1}{2\pi}||[\phi_j,\pi_j]||^2\\
\leq& \frac{1}{2\pi}\int_X Tr(\sqrt{-1}\Lambda \Theta_j \pi_j)dvol\\
=&\frac{1}{2\pi}\int_X Tr(\mathfrak{a}\pi_j)dvol+\frac{1}{2\pi}\int_X(Tr(\sqrt{-1}\Lambda \Theta_j-\mathfrak{a}) \pi_j)dvol\\
\end{split}
\end{equation*}
Still by linear algebra(cf. the material under the proof of Lemma
2.20 in \cite{DW2004}), $Tr(\mathfrak{a}\pi_j) \leq \sum_{i \leq r}
\lambda_i$. Let $j \rightarrow \infty$, the last term tends to zero
,finishing the proof.
\end{proof}

%The problem lies in that $d_A+\phi+\phi^{*H}$ is not a connection
%for twist Higgs bundle, so we can not use the use Chern-Weil theory
%directly. Instead, we consider the connection $d_A$ and the
%curvature $F_A$.

\noindent \textbf{Remark:} Recall that in the untwist case, Simpson
use the connection $D''=d_A''+\phi$(this is not the (0,1) component
$\nabla''$ of the non-unitary connection $\nabla$). The Chern-Weil
formula reads
\begin{equation*}
deg(S)=\frac{1}{2\pi}\int_X Tr(\sqrt{-1}\Lambda R-|D''\pi|^2)dvol.
\end{equation*}
Notice that $|D''\pi|^2=|d_A''\pi|^2+|\phi \pi|^2$ for $d_A''$ is a
(0,1)-form and $\phi$ is a (1,0)-form. So the operator $D''$ is in
effect split and we can threat them independently, this is why the
results in untwist case can be transported to the twist case(Reader
could also refer to \cite{Wil06morse} for the symplectic geometry
interpretation).

Recall the partial ordering of HN types of Higgs bundle
$(E,\bar{\partial}, \phi)$, by the induction on the length of the HN
filtration(cf. \cite{DW2004}), we have:

\begin{proposition}\label{pr:nondecreasingHN}
Let $(A_t, \phi_t)$ be the solution along the YMH flow on a bundle
$(E,H)$ of rank $n$ with limit $(A_\infty, \phi_\infty)$. Let
$\overrightarrow{\mu}_0 = (\mu_1,\dots, \mu_n)$ be the HN type of
$(E, d_{A_0}'', \phi_0)$, and let $\overrightarrow{\lambda}_\infty =
(\lambda_1,\dots, \lambda_n)$ be the type of $(E, d_{A_\infty}'',
\phi_\infty)$. Then $\overrightarrow{\mu}_0 \leq
\overrightarrow{\lambda}_\infty$.
\end{proposition}

This is equivalent to say that the HN type is non-decreasing. Recall
that YMH functional is non-increasing along the gradient flow, we
get following easy corollary generalizes a result in \cite{AB83} to
Higgs bundle:

\begin{corollary}
Let $\overrightarrow{\mu}$ be the HN type of $(E, \bar{\partial},
\phi)$. For any Hermitian metric $H$, denote $A$ the unitary
connection, then $\text{YMH}(A,\phi) \geq 2\pi \sum_{i=1}^n
\mu_i^2$, and the equality holds iff $H$ is the split Hermitian
Yang-Mills metric.
\end{corollary}

This corollary asserts that the HN type can be viewed as a lower
bound of the YMH functional.

\subsection{Weighted YMH functionals}

Notice that
$\text{YMH}(\overrightarrow{\mu})=\text{YMH}(\overrightarrow{\lambda})$
is only a necessary condition for
$\overrightarrow{\mu}=\overrightarrow{\lambda}$. To distinguish
different HN types, we need more functionals. On Riemann surface,
people often use convex functionals to detect the HN type, for
instance \cite{AB83} the vector bundle case and \cite{Wil06morse}
the Higgs bundle case. On K{\"a}hler surface, Daskalopoulos and
Wentworth(\cite{DW2004}) restrict themselves to a subclass of convex
functionals, namely weighted Yang-Mills functionals. Here we follow
their idea, apply this method to the Higgs case(see also
\cite{LZ2011}). Let $\mathfrak{u}(n)$ denote the Lie algebra of the
unitary group $U(n)$. Fix a real number $\alpha \geq 1$. Then for
$\mathfrak{a} \in \mathfrak{u}(n)$, a skew hermitian matrix with
eigenvalues $\sqrt{-1}\lambda_1, \dots, \sqrt{-1}\lambda_n$, let
$\psi_\alpha(\mathfrak{a})=\sum_{j=1}^n |\lambda_j|^\alpha$. By
Prop. 12.16 in \cite{AB83}, $\psi_\alpha$ is a convex function on
$\mathfrak{u}(n)$. Moreover, for a given number $N$, define:

\begin{equation*}
\text{YMH}_{\alpha,N}(A,\phi)=\int_X \psi_\alpha(\Lambda
\Theta+\sqrt{-1}N Id_E)dvol.
\end{equation*}
Take the convention
$\text{YMH}_\alpha(A,\phi)=\text{YMH}_{\alpha,0}(A,\phi)$, and
notice that $\text{YMH} = \text{YMH}_2$ is the ordinary YMH
functional. We make a slight abuse of notation, setting
\begin{equation*}
\text{YMH}_{\alpha,N} (\overrightarrow{\mu}) = \text{YMH}_\alpha
(\overrightarrow{\mu} + N) =
2\pi\psi_\alpha(\sqrt{-1}(\overrightarrow{\mu} + N))
\end{equation*}
where $\overrightarrow{\mu}+N = (\mu_1+N, \dots, \mu_n +N)$ is
identified with the diagonal matrix $diag((\mu_1+N, \dots, \mu_n
+N)$.

Following lemma reveal the connection between weighted YMH
functional and the approximate critical Hermitian structure will be
studied in next subsection.

\begin{lemma}
The functional $\mathfrak{a} \rightarrow (\int_X \psi_\alpha(\mathfrak{a})dvol)^\frac{1}{\alpha}$ defines a norm on $L^\alpha(\mathfrak{u}(E))$ which is
equivalent to the $L^\alpha$ norm $(\int_X (-Tr \mathfrak{a}\cdot\mathfrak{a}^*  )^\frac{\alpha}{2} dvol)^\frac{1}{\alpha}$.
\end{lemma}

\begin{proof}
\begin{equation*}
\frac{1}{C}(\sum_{i=1}^n|\lambda_i|^2)^\frac{\alpha}{2} \leq \frac{1}{C}( \sum_{i=1}^n|\lambda_i|)^\alpha \leq \sum_{i=1}^n|\lambda_i|^\alpha \leq C( \sum_{i=1}^n|\lambda_i|)^\alpha \leq C'( \sum_{i=1}^n|\lambda_i|^2)^\frac{\alpha}{2}.
\end{equation*}
\end{proof}

Now we focus on the relation between the weighted YMH functional and the HN type. Similar with the usual YMH functional, we have,

\begin{proposition}\label{pr:nonincreasingWYMH}
Let $(A_t, \phi_t)$ be a solution of the gradient flow. Then for any
$\alpha \geq 1$ and any $N$, $t \rightarrow
\text{YMH}_{\alpha,N}(A_t,\phi_t)$ is nonincreasing.
\end{proposition}

\begin{proposition}
Let $(A_\infty, \phi_\infty)$ be a limit of $(A_t, \phi_t)$, where
$(A_t, \phi_t)$ is a solution to Equation \eqref{eq:gradientflow}.
Then for any $\alpha \geq 1$ and any $N$, $\lim_{t \rightarrow
\infty} \text{YMH}_{\alpha,N}(A_t, \phi_t) =
\text{YMH}_{\alpha,N}(A_\infty, \phi_\infty)$.
\end{proposition}

The proof is parallel with the proof in \cite{DW2004} for the vector bundle case. The key point of introducing such kind of functional is that they can distinguish different HN type.

\begin{proposition}\label{pr:weightedYMH}
(1) If  $\overrightarrow{\mu} \leq \overrightarrow{\lambda}$, then $\psi_\alpha(\sqrt{-1}\overrightarrow{\mu}) \leq \psi_\alpha(\sqrt{-1}\overrightarrow{\lambda})$ for all $\alpha \geq 1$.
(2) Assume $\mu_n \geq 0$ and $\lambda_n \geq 0$. If $\psi_\alpha(\sqrt{-1}\overrightarrow{\mu}) = \psi_\alpha(\sqrt{-1}\overrightarrow{\lambda})$ for all $\alpha \geq 1$. then $\overrightarrow{\mu} = \overrightarrow{\lambda}$.
\end{proposition}

\begin{proof}
(1) follows from \cite{AB83}, Equation 12.5. For (2), consider
$f(\alpha) = \psi_\alpha(\sqrt{-1}\overrightarrow{\mu})$ and
$g(\alpha) = \psi_\alpha(\sqrt{-1}\overrightarrow{\lambda})$ as
functions of $\alpha$. As complex valued functions, $f$ and $g$
clearly have analytic extensions to $\mathbb{C} \backslash \{\alpha
\leq 0\}$. Suppose that $f(\alpha) = g(\alpha)$ for all $\alpha \geq
1$. Then by analyticity, $f(\alpha) = g(\alpha)$ for all $\mathbb{C}
\backslash \{\alpha \leq 0\}$. If $\overrightarrow{\mu} \neq
\overrightarrow{\lambda}$, then there is some $k$, $1 \leq k \leq
n$, such that $\mu_i = \lambda_i$ for $i < k$, and $\mu_k \neq
\lambda_k$; say, $\mu_k > \lambda_k$.

Then for any $\alpha > 0$:
\begin{equation*}
(\frac{\mu_k}{\lambda_k})^\alpha \leq \sum_{i=k}^n (\frac{\mu_i}{\lambda_k})^\alpha =\sum_{i=k}^n (\frac{\lambda_i}{\lambda_k})^\alpha \leq n,
\end{equation*}
where the middle equality follows from $f(\alpha) = g(\alpha)$ and $\mu_i = \lambda_i$ for $i < k$. Letting $\alpha \rightarrow \infty$, we obtain a contradiction.
\end{proof}

\subsection{Approximate critical Hermitian structure}

To prove the gradient flow preserving HN type, we need equality
$\text{YMH}_{\alpha,N}(A_\infty,\phi_\infty)=\text{YMH}_{\alpha,N}(\overrightarrow{\mu}_0)$.
Although the weighted YMH functional is non-increasing and the HN
type is non-decreasing, there still may be some jumping phenomenon
illustrated in following figure
\begin{figure}[h]
\includegraphics{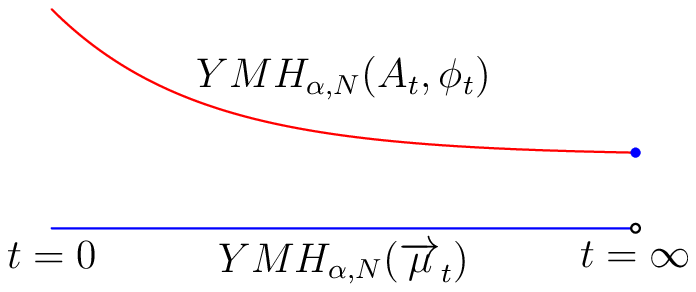}
\end{figure}

\noindent We need another tool, namely approximate critical
Hermitian structure on Higgs bundle, to show there is in fact no gap
between $\text{YMH}_{\alpha,N}(A_\infty,\phi_\infty)$ and
$\text{YMH}_{\alpha,N}(\overrightarrow{\mu}_0)$.

The following is the Higgs version of definition introduced in
\cite{DW2004} for vector bundle. Fix a Higgs bundle $(E,
\bar{\partial},\phi)$ and a Hermitian metric $H$. Let
$\{E_i\}_{i=1}^l$ be the HN filtration. Associated to each $E_i$ the
unitary projection $\pi^H_i$ from $E$ to $E_i$. For convenience, we
set $\pi^H_0 = 0$. The $\pi^H_i$ are bounded $L^2$ Hermitian
endomorphisms. Then the Harder-Narasimhan projection, $\Psi^{hn}
(E,\bar{\partial},\phi,H)$, is defined by $\sum_{i=1}^l \mu_i
(\pi^H_i-\pi^H_{i-1})$, which is a bounded $L^2$ Hermitian
endomorphism.

\begin{definition}
Fix $\delta > 0$ and $1 \leq p \leq \infty$. An
$L^p$-$\delta$-approximate critical Hermitian structure on a Higgs
bundle E is a smooth metric $H$ such that
\begin{equation*}
||\sqrt{-1}\Lambda \Theta_{(\bar{\partial},\phi,H)}-\Psi^{hn} (E,\bar{\partial},\phi,H)||_{L^p} \leq \delta.
\end{equation*}
\end{definition}

\begin{theorem}\label{th:approExist}
For any $\delta > 0$, there is an $L^\infty$-$\delta$-approximate
critical Hermitian structure $H$ on $(E,\bar{\partial},\phi)$.
\end{theorem}

\begin{proof}

First, by the equivalence of holomorphic structures $\bar{\partial}$
and the unitary connections $A$, it suffices to show that for a
fixed Hermitian metric $H$ there is a smooth complex gauge
transformation $g$ preserving the HN filtration such that:
\begin{equation}\label{eq:approGauge}
||\sqrt{-1}\Lambda \Theta_{(g(\bar{\partial},\phi),H)}-\Psi^{hn}
(g(\bar{\partial},\phi),H)||_{L^\infty} \leq \delta.
\end{equation}

Next, for semistable $E$ (i.e. the length 1 case), the result
follows by the convergence $||\sqrt{-1}\Lambda
\Theta_{(A_t,\phi_t)}-\mu(E)Id_E||_{L^\infty} \rightarrow 0$, where
$(A_t,\phi_t)$ is a solution to the gradient flow Equation
\eqref{eq:gradientflow} with any initial condition (cf.
\cite{Don85anti}, Cor.25, here the Higgs case is similar). With this
understood, choose $\delta'$-approximate metrics, where $0 <
\delta'<< \delta$, on the semistable quotients $Q_i$ of the HN
filtration of $E$ to fix a metric $H$ on $E = Q_1\oplus\cdots\oplus
Q_l$. Then by appropriately scaling the extension classes: $0
\rightarrow E_{i-1} \rightarrow E_i \rightarrow Q_i \rightarrow 0$,
one finds a complex gauge transformation satisfying
\eqref{eq:approGauge}. We omit the details.
\end{proof}

By the equivalence of $L^p$ norm and the weighted YMH functional, we have

\begin{corollary}\label{co:approStr}
Let $E$ be a Higgs bundle of HN type $\overrightarrow{\mu}_0$. There is $\alpha_0 > 1$ such that the following holds: given any $\delta > 0$ and any $N$, there is a Hermitian metric $H$ on $E$ such that
\begin{equation*}
\text{YMH}_{\alpha,N}(E,\bar{\partial},\phi) \leq
\text{YMH}_{\alpha,N}(\overrightarrow{\mu}_0) + \delta, \ \
\text{for all} \ \ 1 \leq \alpha \leq \alpha_0.
\end{equation*}
\end{corollary}

\subsection{Proof of theorem}

With these preparation in hand, we can prove Theorem \ref{th:preserveHN} by using the approximate critical Hermitian structure on Higgs bundle to eliminate the possibility of the jumping phenomenon,
\begin{lemma}
\begin{equation*}
\lim_{t \rightarrow \infty} \text{YMH}_{\alpha,N}(A_t,
\phi_t)=\text{YMH}_{\alpha,N}(\overrightarrow{\mu}_0).
\end{equation*}
\end{lemma}

The proof is divided into two parts.

\noindent \textbf{Step1}, for fixed $\alpha$ and fixed $N$, define $\delta_0 > 0$ by:
\begin{equation}\label{eq:definitionDelta0}
 2 \delta_0 + \text{YMH}_{\alpha,N}(\overrightarrow{\mu}_0) = \min\{\text{YMH}_{\alpha,N}(\overrightarrow{\mu}) : \text{YMH}_{\alpha,N}(\overrightarrow{\mu}) > \text{YMH}_{\alpha,N}(\overrightarrow{\mu}_0)\},
\end{equation}
where $\overrightarrow{\mu}$ runs over all possible HN types of
Higgs bundles on $X$ with the rank of $(E,\bar{\partial},\phi_0)$.
For $\overrightarrow{\mu}$ is discrete, $\delta_0$ always exists. By
corollary \ref{co:approStr}, consider metrics $H$ on $E$ with
associated connection $A_0 = (\bar{\partial},H)$ satisfying:
\begin{equation}\label{eq:initialcon0}
\text{YMH}_{\alpha,N}(A_0,\phi_0) \leq
\text{YMH}_{\alpha,N}(\overrightarrow{\mu}_0) + \delta_0.
\end{equation}
Let $(A_\infty, \phi_\infty)$ be the limit along the flow with
initial condition $(A_0,\phi_0)$. Then combining
Prop.\ref{pr:nondecreasingHN}, Prop.\ref{pr:weightedYMH} (1), and
Prop.\ref{pr:nonincreasingWYMH}, we have:
\begin{equation*}
\text{YMH}_{\alpha,N}(\overrightarrow{\mu}_0) \leq
\text{YMH}_{\alpha,N}(A_\infty,\phi_\infty) \leq
\text{YMH}_{\alpha,N}(A_0,\phi_0)\leq
\text{YMH}_{\alpha,N}(\overrightarrow{\mu}_0)+\delta_0 .
\end{equation*}
By Equation \eqref{eq:definitionDelta0} the definition of
$\delta_0$, we must have
$\text{YMH}_{\alpha,N}(A_\infty,\phi_\infty)=
\text{YMH}_{\alpha,N}(\overrightarrow{\mu}_0)$. This shows that the
result holds for initial conditions satisfying (4.4).

\noindent \textbf{Step2}, in the following, we want to show for any
initial data, after long enough time,
$\text{YMH}_{\alpha,N}(A_t,\phi_t)$ will approach
$\text{YMH}_{\alpha,N}(\overrightarrow{\mu}_0)$ sufficient close,
then reducing the problem to step1. More precisely, let us denote by
$(A^H_t, \phi_t)$ the solution to the YMH flow at time $t$ with
initial condition $A_0=(\bar{\partial},H)$. We are going to prove
that for any $H$ and any $\delta > 0$, there is $T \geq 0$ such
that:
\begin{equation}\label{eq:initialconH}
\text{YMH}_{\alpha,N}(A_t,\phi_t) <
\text{YMH}_{\alpha,N}(\overrightarrow{\mu}_0) + \delta, \ \
\text{for all} \ \  t \geq T.
\end{equation}
Without loss of generality, assume $0 < \delta \leq \delta_0/2$. Let
$\mathcal{H}_\delta$ denote the set of smooth hermitian metrics $H$
on $E$ with the property that \eqref{eq:initialconH} holds for
$(A_t^H, \phi)$ and some $T$. We employ open and closeness argument
to show $\mathcal{H}_\delta$ containing all the smooth Hermitian
metric.

First, $\mathcal{H}_\delta$ is non-empty. Indeed, any metric
satisfying Equation \eqref{eq:initialcon0} is in
$\mathcal{H}_\delta$, and according to Theorem \ref{th:approExist},
we may always find such kind of metric.

Second, $\mathcal{H}_\delta$ is open. This is an easy consequence of
the continuous dependence on the initial data  in finite time of the
flow in the $C^\infty$ topology.

Third, $\mathcal{H}_\delta$ is closed. Let $H_j$ be a sequence of
smooth Hermitian metrics on $E$ such that each $H_j \in
\mathcal{H}_\delta$, and suppose $H_j \rightarrow K$, in the
$C^\infty$ topology, for some metric $K$. We want to show that $K
\in \mathcal{H}_\delta$

Since $H_j \in \mathcal{H}_\delta$, we have a sequence $T_j$ such
that for all $t \geq T_j$:
\begin{equation*}
\text{YMH}_{\alpha,N}(A^{H_j}_t,\phi_t) \leq
\text{YMH}_{\alpha,N}(A^{H_j}_{T_j},\phi_{T_j})
\text{YMH}_{\alpha,N}(\overrightarrow{\mu}_0) + \delta.
\end{equation*}
We may find a sequence $t_j \geq T_j$, s.t. $(A^{H_j}_{t_j}, \phi^{H_j}_{t_j}) \rightarrow (A^{(1)}_\infty, \phi^{(1)}_\infty)$ in $L^p$ for all $p$. At another hand, $(A^K_{t_j}, \phi^K_{t_j}) \rightarrow (A^{(2)}_\infty, \phi^{(2)}_\infty)$.

Using Higgs version of Donaldson's functional(cf. \cite{Don85anti}, or the Higgs case \cite{LZ2011}), it is not difficult to show
\begin{equation*}
(A^{(1)}_\infty, \phi^{(1)}_\infty)=(A^{(2)}_\infty, \phi^{(2)}_\infty).
\end{equation*}
Thus set $(A_\infty, \phi_\infty)=(A^{(1)}_\infty,
\phi^{(1)}_\infty)=(A^{(2)}_\infty, \phi^{(2)}_\infty)$, then
\begin{equation*}
\lim_{j \rightarrow
\infty}\text{YMH}_{\alpha,N}(A^{H_j}_{T_j},\phi_{T_j})= \lim_{j
\rightarrow \infty}\text{YMH}_{\alpha,N}(A^K_{t_j},
\phi^K_{t_j})=\text{YMH}_{\alpha,N}(A_\infty, \phi_\infty).
\end{equation*}

Hence, for $j$ sufficiently large:
\begin{equation*}
\begin{split}
&\text{YMH}_{\alpha,N}(A^K_{t_j}, \phi^K_{t_j}) \leq \text{YMH}_{\alpha,N}(A_\infty, \phi_\infty)+\delta\\
=&\lim_{j \rightarrow \infty}\text{YMH}_{\alpha,N}(A^{H_j}_{T_j},\phi_{T_j})+\delta \leq  \text{YMH}_{\alpha,N}(\overrightarrow{\mu}_0) + 2\delta\\
\leq & \text{YMH}_{\alpha,N}(\overrightarrow{\mu}_0) + \delta_0.
\end{split}
\end{equation*}
Therefore, $K \in \mathcal{H}_\delta$.

Since the space of smooth metrics is connected, we conclude that
every metric is in $\mathcal{H}_\delta$, and \eqref{eq:initialconH}
holds for all $\delta > 0$ and all metric $H$. In particular, we can
choose $\delta \leq \delta_0$ and conclude that
\begin{equation*}
\lim_{t\rightarrow\infty} \text{YMH}_{\alpha,N}(A^H_t,\phi^H_t) =
\text{YMH}_{\alpha,N} (\overrightarrow{\mu}_0), \ \text{for any} \
H.
\end{equation*}
Since the choice of $N$ was arbitrary, the proof of Lemma is
complete, and Theorem \ref{th:preserveHN} follows.

\section{Convergence to the graded object}

We will finish the proof of Theorem \ref{th:convergeGR} in this
section.

%In last section, we had showed that the HN type is preserved by the flow.

The proof mainly follows the argument in the proof of Theorem 5.3 in
\cite{Wil06morse}. For $d_{A_j}+\phi_j+\phi^{*H}_j$ is no longer a
connection, there are two major differences. One is that we should
modify the Chern-Weil formula, another is using the unitary
connection $d_A$ to build the Sobolev space to get the convergence.

Theorem \ref{th:preserveHN} already shows that the type of the
Harder-Narasimhan filtration is preserved in the limit. We want to
show that the destabilising Higgs sub-bundles in the
Harder-Narasimhan filtration along the gradient flow also converge
to the destabilising Higgs sub-bundles of the limiting Higgs pair.
This fact will lead us to the construction of the holomorphic
morphism between $Gr^{hns}(E,d_{A_0}'',\phi_0)$ to
$(E,d_{A_\infty}'',\phi_\infty)$. In the following we use the
projection $\pi : E \rightarrow E$ to denote the sub-bundle
$\pi(E)$.

\begin{proposition}\label{pr:limitHNfiltration}
Let $\{\pi^{(i)}_t\}$ be the HN filtration of a solution $(A_t, \phi_t)$ to the gradient flow equations \eqref{eq:gradientflow}, and let $\{\pi^{(i)}_\infty\}$ be the HN filtration of the limit $(A_\infty, \phi_\infty)$. Then there exists a subsequence $\{t_j\}$ such that $\pi^{(i)}_{t_j}\rightarrow \pi^{(i)}_\infty$ in $L^2$ for all $i$.
\end{proposition}
To prove this we need the following lemmas.

\begin{lemma}
$||d_{A_t}''(\pi_t^{(i)})||_{L^2} \rightarrow 0$ and $||[\phi_t,\pi_t^{(i)}]||_{L^2} \rightarrow 0$
\end{lemma}

\begin{proof}
The Chern-Weil formula \eqref{eq:chernweil} shows that
\begin{equation}\label{eq:degreeSequence}
deg(\pi^{(i)}_t)=\frac{\sqrt{-1}}{2\pi}\int_X tr(\pi^{(i)}_t \Lambda(F_{A_t} + [\phi_t, \phi_t^{*H}])-||d_{A_t}''(\pi^{(i)})||_{L^2}^2-||[\phi_t,\pi_t^{(i)}]||_{L^2}^2
\end{equation}
Along the finite-time flow $d_i = deg(\pi^{(i)}_t)$ is fixed,
therefore we can re-write Equation \eqref{eq:degreeSequence}
\begin{equation}\label{eq:rewriteDegree}
\begin{split}
&||d_{A_t}''(\pi^{(i)})||_{L^2}^2-||[\phi_t,\pi_t^{(i)}]||_{L^2}^2\\
=& -d_i + \frac{\sqrt{-1}}{2\pi}\int_X tr(\pi^{(i)}_t \Lambda(F_{A_\infty} + [\phi_\infty, \phi_\infty^{*H}])\\
+&\frac{\sqrt{-1}}{2\pi}\int_X tr(\pi^{(i)}_t \Lambda(F_{A_t} + [\phi_t, \phi_t^{*H}]-F_{A_\infty} - [\phi_\infty, \phi_\infty^{*H}])\\
\end{split}
\end{equation}
Since $F_{A_t} + [\phi_t, \phi_t^{*H}] \rightarrow F_{A_\infty} +
[\phi_\infty, \phi_\infty^{*H}]$ in $C^\infty$ topology and
$\pi^{(i)}_t$ uniformly bounded in $L^2$ (for it is a projection)
then the last term in \eqref{eq:rewriteDegree} converges to zero.
Let $\overrightarrow{\mu}$ be the HN type of $(E,d_{A_\infty}'',
\phi_\infty)$. Since $(d_{A_\infty}'', \phi_\infty)$ is a critical
point of YMH then we also have
\begin{equation}
\frac{\sqrt{-1}}{2\pi}\int_X tr(\pi^{(i)}_t \Lambda(F_{A_\infty} + [\phi_\infty, \phi_\infty^{*H}]) \leq \sum_{k \leq rank(\pi^{(i)}_\infty)} \mu_k=d_i.
\end{equation}
Combining all of these results, we see that $||d_{A_t}''(\pi_t^{(i)})||_{L^2} \rightarrow 0$ and $||[\phi_t,\pi_t^{(i)}]||_{L^2} \rightarrow 0$.
\end{proof}

In particular, this lemma shows that $||\pi_t^{(i)}||^{H^1} \leq C$ and so there exists some $\tilde{\pi}^{(i)}_\infty$ and a subsequence $t_j$ such that $\pi_t^{(i)}\rightarrow \tilde{\pi}^{(i)}_\infty$ weakly in $H^1$ and strongly in $L^2$.

\begin{lemma}
$||d_{A_\infty}''(\pi_\infty^{(i)})||_{L^2}=0$ and $||[\phi_\infty,\pi_\infty^{(i)}]||_{L^2}=0$
\end{lemma}
\begin{proof}
For $||d_{A_\infty}''(\pi_{t_j}^{(i)})||_{L^2} \leq ||d_{A_\infty}''(\pi_{t_j}^{(i)})-d_{A_{t_j}}''(\pi_\infty^{(i)})||_{L^2} +||d_{A_{t_j}}''(\pi_{t_j}^{(i)})||_{L^2}$, by the continuously convergence of the gradient flow and previous lemma, $||d_{A_\infty}''(\pi_{t_j}^{(i)})||_{L^2} \rightarrow 0$. Since $\pi^{(i)}_{t_j} \rightarrow \tilde{\pi}^{(i)}_\infty$ weakly in $H^1$ then $||d_{A_\infty}''(\pi_\infty^{(i)})||_{L^2}=0$. The proof of $||[\phi_\infty,\pi_\infty^{(i)}]||_{L^2}=0$ is similar.
\end{proof}

This lemma implies that $\pi_\infty^{(i)}$ is indeed a $\phi_\infty$
invariant split Higgs bundle.
\begin{lemma}
$deg(\tilde{\pi}^{(i)}_\infty)=deg(\pi^{(i)}_\infty)$.
\end{lemma}

\begin{proof}
The previous lemma and Equation \eqref{eq:degreeSequence} show that
\begin{equation*}
\begin{split}
deg(\tilde{\pi}^{(i)}_\infty)=&\frac{\sqrt{-1}}{2\pi}\int_X tr(\tilde{\pi}^{(i)}_t \Lambda(F_{A_\infty} + [\phi_\infty, \phi_\infty^{*H}])\\
=&\lim_{j\rightarrow\infty}\frac{\sqrt{-1}}{2\pi}\int_X tr(\pi^{(i)}_{t_j} \Lambda(F_{A_{t_j}} + [\phi_{t_j}, \phi_{t_j}^{*H}])\\
=&\lim_{j\rightarrow\infty} (||d_{A_{t_j}}''(\pi_{t_j}^{(i)})||_{L^2}^2+||[\phi_{t_j},\pi_{t_j}^{(i)}]||_{L^2}^2)+ deg(\pi_{t_j}^{(i)})\\
=&deg(\pi^{(i)}_\infty)
\end{split}
\end{equation*}
where in the last step we use the result of Theorem
\ref{th:preserveHN} that the type of HN filtration is preserved in
the limit.
\end{proof}

Notice that $\pi^{(i)}_{t_j}$ are all orthogonal projection with
constant rank, the limit $\tilde{\pi}^{(i)}_\infty$ must has the
same rank. For $i = 1$, $\pi^{(1)}_\infty$ is the maximal
destabilising semistable Higgs sub-bundle of $(E, d_{A_\infty}'',
\phi_\infty)$, which is the unique Higgs sub-bundle of this degree
and rank. Therefore $\pi^{(1)}_\infty=\tilde{\pi}^{(1)}_\infty$.
Proceeding by induction on the HN filtration as in \cite{DW2004}, we
can show all $\pi^{(i)}_\infty$ and $\tilde{\pi}^{(i)}_\infty$ are
the same, then Proposition \ref{pr:limitHNfiltration} follows. This
means that not only the type of HN filtration but also the HN
filtration itself is preserved by the gradient flow.

The same argument applies to the Seshadri filtration of a semistable
Higgs bundle, except that because of the lack of uniqueness of the
Seshadri filtration we can only conclude that the degree and rank of
the limiting sub-bundle are the same. There must be another way to
identify the stable Higgs bundles in the HNS filtration and
$(E,d_{A_\infty}'', \phi_\infty)$.

Fix $S$ to be the first term in the Harder-Narasimhan-Seshadri
filtration of $(E,d_{A_0}'',\phi_0)$. Following
Donaldson\cite{Don85anti}, if there is a nontrivial holomorphic map
from $S$ to $(E,d_{A_\infty}'', \phi_\infty)$, we can apply the
basic principle that a nontrivial holomorphic map between stable
bundles of the same rank and degree must be an isomorphism. Denote
$(A_{t_j}, \phi_{t_j})$ by $(A_j,\phi_j)$, and let $g_j$ be the
complex gauge transformation such that $(A_j,\phi_j)=g_j(A_0,
\phi_0)$.  Let $f_0: S \rightarrow E$ be the $\phi_0$-invariant
holomorphic inclusion, define the map $f_j : S_j \rightarrow E$ by
$f_j=g_j \circ f_0$. It is easy to check that $f_j$ is a
$\phi$-invariant holomorphic bundle map from $(S, d_{A_0}'',\phi_0)$
to $(E, d_{A_j}'',\phi_j)$, here $\phi$-invariant means $f_j \circ
\phi_0 = \phi_j \circ f_j$. Then we have

\begin{lemma}
Up to a subsequence, $f_j$ converges in $C^\infty$ to some nonzero $\phi$-invariant
holomorphic map $f_\infty$.
\end{lemma}

Reader can refer to \cite{LZ2011} for the details of the proof.

If $\pi_j$ denotes the projection to $f_j(S)$, then as mentioned under the proof of Proposition \ref{pr:limitHNfiltration}, $\pi_j \rightarrow \pi_\infty$ weakly in $H^1$ and strongly in $L^2$, where $\pi_\infty$ is a subbundle of the same rank and degree as $S$. Denote $S_\infty = \pi_\infty(E)$, then $f_\infty$ is in effect a map from $(S, d_{A_0}'',\phi_0)$ to $(S_\infty, d_{A_\infty}'',\phi_\infty)$.

A prior, $f_\infty: S \rightarrow S_\infty$ could be any bad, but we
have the following lemma completely analogous to the proof of
(V.7.11) in \cite{Ko87differential} for holomorphic bundles and so
the proof is omitted.

\begin{lemma}
Let $(S_1, \bar{\partial}_1, \phi_1)$ be a stable Higgs bundle, and let $(S_2, \bar{\partial}_2, \phi_2)$ be a semistable Higgs bundle over a compact Riemann surface X. Also suppose that $\frac{deg(S_1)}{rank(S_1)}= \frac{deg(S_2)}{rank(S_2)}$, and let $f : S_1 \rightarrow S_2$ be a holomorphic map satisfying $f \circ \phi_1 = \phi_2 \circ f$. Then either $f = 0$ or $f$ is injective.
\end{lemma}

%$(f_\infty(S), d_{A_\infty}'', \phi_\infty)$ will be a Higgs subbundle of $(E, d_{A_\infty}'', \phi_\infty)$.

For the gradient flow preserve HN filtration, $(S_\infty,
d_{A_\infty}'', \phi_\infty)$ still lies in the maximal
destablishing semistable Higgs subbundle and has the highest slope
in the HN filtration. Hence $(S_\infty, d_{A_\infty}'',
\phi_\infty)$ can not admit any subbundle with higher slope, must be
semi-stable. By above lemma, $f_\infty$ is injective. But $S_\infty$
has the same rank with $S$, then $f_\infty$ must be an isomorphism,
and $S_\infty=f_\infty(S)$ is a stable factor in the split Higgs
bundle $(E, d_{A_\infty}'', \phi_\infty)$. Relabel $f_\infty$ as
$f_\infty^{(1)(1)}$ to represent the highest factor in the HNS
filtration. Now we construct a isomorphism form $S$ in
$Gr^{hns}(E,d_{A_0}'',\phi_0)$ to $S_\infty$ in
$(E,d_{A_\infty}'',\phi_\infty)$. To build the entire isomorphism
from $Gr^{hns}(E,d_{A_0}'',\phi_0)$ to
$(E,d_{A_\infty}'',\phi_\infty)$, we make induction on the length of
the HNS filtration. Let $Q=E/S$, we have
$Gr^{hns}(E,d_{A_0}'',\phi_0) = S \oplus
Gr^{hns}(E,d_{A_0}''{^Q},\phi_0^Q)$, and
$(E,d_{A_\infty}'',\phi_\infty)=S_\infty \oplus Q_\infty$. Follow
the discussion in \cite{LZ2011}, we can prove that $Q_\infty \cong
Gr^{hns}(E,d_{A_0}''{^Q},\phi_0^Q)$. Similarly, there is a map
$f_\infty^{(1)(2)}$ from $S' \subset Q$ to $S'_\infty \subset Q$.
Repeat this procedure, there exists a isomorphism
$\{f_\infty^{(i)(j)}\}$ identifying $Gr^{hns}(E,d_{A_0}'',\phi_0)$
with $(E,d_{A_\infty}'',\phi_\infty)$.

Finally, the limit $\{f_\infty^{(i)(j)}\}$ exists along the flow
independently of the subsequence chosen, then we complete the proof
of Theorem \ref{th:convergeGR}.

\noindent \textbf{Remark:} The fixed Hermitian metric $H$ on $E$ for
the YMH flow is arbitrary, but Theorem \ref{th:convergeGR} informs
us the limit does not depend on the choice of metric.

\bibliographystyle{alpha}
\bibliography{higgs}

\begin{thebibliography}{BDGPW95}

\bibitem[AB83]{AB83}
M.F. Atiyah and R.~Bott.
\newblock The {Y}ang-{M}ills equations over riemann surfaces.
\newblock {\em Philosophical Transactions of the Royal Society of London.
  Series A, Mathematical and Physical Sciences}, pages 523--615, 1983.

\bibitem[BDGPW95]{bradlow1995stable}
S.~Bradlow, G.~Daskalopoulos, O.~Garc{\i}a-Prada, and R.~Wentworth.
\newblock {\em Stable augmented bundles over Riemann surfaces}, chapter Vector
  Bundles in Algebraic Geometry, pages 15--77.
\newblock Cambridge University Press, Cambridge, 1995.

\bibitem[CJ12]{CollinsJacob2012YangMills}
T.~{Collins} and A.~{Jacob}.
\newblock {On the bubbling set of the Yang-Mills flow on a compact K{\"a}hler
  manifold}.
\newblock {\em ArXiv e-prints}, June 2012.

\bibitem[Don85]{Don85anti}
S.K. Donaldson.
\newblock Anti self-dual {Y}ang-{M}ills connections over complex algebraic
  surfaces and stable vector bundles.
\newblock {\em Proceedings of the London Mathematical Society}, 50(1):1--26,
  1985.

\bibitem[DW04]{DW2004}
G.~Daskalopoulos and R.~Wentworth.
\newblock Covergence properties of the yang-mills flow on k{\"a}hler surfaces.
\newblock {\em J. Reine Angew. Math.}, 575:69--99, 2004.

\bibitem[Hit87]{Hit87self}
N.~Hitchin.
\newblock The self-duality equations on a {R}iemann surface.
\newblock {\em Proc. London Math. Soc.}, 55(3):59--126, 1987.

\bibitem[HT03]{HT03mirror}
T.~Hausel and M.~Thaddeus.
\newblock Mirror symmetry, {L}anglands duality, and the {H}itchin system.
\newblock {\em Inventiones Mathematicae}, 153(1):197--229, 2003.

\bibitem[Kob87]{Ko87differential}
S.~Kobayashi.
\newblock {\em Differential geometry of complex vector bundles}.
\newblock Iwanami Shoten, 1987.

\bibitem[LZ11]{LZ2011}
Jiayu Li and Xi~Zhang.
\newblock The gradient flow of higgs pairs.
\newblock {\em Journal of the European Mathematical Society}, 13(5):1373--1422,
  2011.

\bibitem[Nit91]{nit91moduli}
N.~Nitsure.
\newblock Moduli space of semistable pairs on a curve.
\newblock {\em Proceedings of the London Mathematical Society}, 3(2):275--300,
  1991.

\bibitem[{Sib}12]{Sibley2012YangMills}
B.~{Sibley}.
\newblock {Asymptotics of the Yang-Mills Flow for Holomorphic Vector Bundles
  Over K{\"a}hler Manifolds: The Canonical Structure of the Limit}.
\newblock {\em ArXiv e-prints}, June 2012.

\bibitem[Sim88]{Sim88constructing}
C.T. Simpson.
\newblock Constructing variations of {H}odge structure using {Y}ang-{M}ills
  theory and applications to uniformization.
\newblock {\em J. Amer. Math. Soc}, 1, 1988.

\bibitem[{Wil}06]{Wil06morse}
G.~{Wilkin}.
\newblock {Morse Theory for the Space of Higgs Bundles}.
\newblock {\em ArXiv Mathematics e-prints}, November 2006.

\bibitem[WZ11]{WZ11twisted}
Yue Wang and Xi~Zhang.
\newblock Twisted holomorphic chains and vortex equations over non-compact
  {K}{\"a}hler manifolds.
\newblock {\em Journal of Mathematical Analysis and Applications},
  373(1):179--202, 2011.

\end{thebibliography}

\

Department of Mathematics, South China University of Technology,
Guangzhou, 510641, P.R.China.

\emph{E-mail Address:} sczhangw@scut.edu.cn

\

\end{document}